\documentclass[10pt]{amsart}
\usepackage{amssymb}
\usepackage{tikz-cd}
\usepackage{amsrefs}
\usepackage{xfrac}
\usepackage{thmtools}
\usepackage{thm-restate}
\usepackage{comment}

\usepackage[french,english]{babel}
\usepackage[T1]{fontenc}

\title[Coordinate systems in Banach spaces and lattices]{Coordinate systems in Banach spaces and lattices\\  (Systèmes de coordonnées dans les espaces et treillis de Banach)}

\author{A.~Avil\'es}
\address{Universidad de Murcia, Departamento de Matem\'aticas, Campus de Espinardo 30100 Murcia, Spain.}
\email{avileslo@um.es}

\author {C.~Rosendal}
\address{Department of Mathematics\\University of Maryland\\4176 Campus Drive - William E. Kirwan Hall\\College Park, MD 20742-4015\\USA}
\email{rosendal@umd.edu}
\urladdr{sites.google.com/view/christian-rosendal/}

\author{M.~A.\ Taylor}
\address{Department of Mathematics\\
ETH Z\"urich, Ramistrasse 101, 8092 Z\"urich, Switzerland.
} \email{mitchell.taylor@math.ethz.ch}

\author{P.~Tradacete}
\address{Instituto de Ciencias Matem\'aticas (CSIC-UAM-UC3M-UCM)\\
Consejo Superior de Investigaciones Cient\'ificas\\
C/ Nicol\'as Cabrera, 13--15, Campus de Cantoblanco UAM\\
28049 Madrid, Spain.}
\email{pedro.tradacete@icmat.es}

\newcommand{\norm}[1]{\lVert#1\rVert}

\newcommand{\NORM}[1]{\Big\lVert#1\Big\rVert}
\newcommand{\BNORM}[1]{\Bigg\lVert#1\Bigg\rVert}

\newcommand{\forkindep}[1][]{\mathop{\mathop{\vcenter{\hbox{\oalign{\noalign{\kern-.3ex}
\hfil$\vert$\hfil\cr\noalign{\kern-.7ex}$\smile$\cr\noalign{\kern-.3ex}}}}}\displaylimits_{#1}}}

\newcommand{\maths}[1]{\[\begin{split}{#1}\end{split}\]}

\newcommand{\conv}[2]{\mathop{\underset{#2}{\overset{#1}\longrightarrow}}}
\newcommand{\nconv}[2]{\mathop{\underset{#2}{\overset{#1}\not\!\!\longrightarrow}}}

\newcommand{\maps}[1]{\mathop{\overset{#1}\longrightarrow}}

\usetikzlibrary{decorations.markings}
\tikzset{negated/.style={
        decoration={markings,
            mark= at position 0.5 with {
                \node[transform shape] (tempnode) {$\backslash$};
            }
        },
        postaction={decorate}
    }
}

\newcommand {\N}{\mathbb N}

\newcommand {\R}{\mathbb R}

\newcommand{\eps}{\varepsilon}

\newcommand{\tom} {\emptyset}

\newcommand{\saa}{\Rightarrow}
\newcommand{\equi}{\Leftrightarrow}

\newcommand {\Del}{ \; \Big| \;}
\newcommand {\del}{ \; \big| \;}

\newcommand {\ku} {\mathcal}

\newcommand{\inv}{^{-1}}
\newcommand {\e} {\exists}
\renewcommand {\a} {\forall}

\theoremstyle{plain}
\newtheorem{thm}{Theorem}[section]
\newtheorem*{theorem*}{Theorem}
\newtheorem{cor}[thm]{Corollary}
\newtheorem{lemme}[thm]{Lemma}
\newtheorem{prop} [thm] {Proposition}
\newtheorem{defi} [thm] {Definition}

\newtheorem{prob}[thm]{Problem}

\theoremstyle{definition}

\newtheorem{rem}[thm]{Remark}
\newtheorem{exa}[thm]{Example}

\definecolor{groen}{rgb}{0,0.5,.7}
\definecolor{gul}{rgb}{0.94,0.8,0}
\definecolor{blaa}{rgb}{0.16,0,0.6}
\definecolor{roed}{rgb}{1,0,0}

\makeindex

\begin{document}


\begin{abstract}
Using Baire category methods from descriptive set theory, we answer several questions from the literature regarding different  notions of bases in Banach spaces and lattices. 

For the case of Banach lattices, our results follow from a general theorem stating that (under the assumption of analytic determinacy),  every $\sigma$-order basis $(e_n)$ for a Banach lattice $X=[e_n]$ is  a uniform basis for $X$ and every uniform basis is Schauder. Moreover, the notions of order and $\sigma$-order bases coincide when $X=[e_n].$

Regarding Banach spaces, we address two problems concerning filter Schauder bases, i.e.,  in which the norm convergence of partial sums is replaced by norm convergence along some appropriate filter on $\N$. We first provide an example of a Banach space admitting such a filter Schauder basis, but no ordinary Schauder basis. Then,  we show that every filter Schauder basis with respect to an analytic filter is also a filter Schauder basis with respect to a Borel filter.
\vspace{0.3cm}

\noindent\textsc{Résumé.} En utilisant des méthodes de catégorie de Baire et de théorie descriptive des ensembles, nous répondons à plusieurs questions de la littérature concernant différentes notions de base dans les espaces et les treillis de Banach.

Dans le cas des treillis de Banach, nos résultats découlent d’un théorème général affirmant que (sous l’hypothèse de la détermination analytique) toute base de $\sigma$-ordre $(e_n)$ d’un treillis de Banach $X=[e_n]$ est une base uniforme pour $X$, et toute base uniforme est de Schauder. De plus, les notions de base d’ordre et de $\sigma$-ordre coïncident lorsque $X=[e_n]$.

Concernant les espaces de Banach, nous traitons deux problèmes concernant les bases de Schauder relatives à un filtre, c’est-à-dire pour lesquelles la convergence en norme des sommes partielles est remplacée par la convergence en norme le long d’un filtre bien choisi sur $\N$. Nous fournissons d’abord un exemple d’espace de Banach admettant une telle base de Schauder relative à un filtre, mais aucune base de Schauder ordinaire. Ensuite, nous montrons que toute base de Schauder relative à un filtre analytique est également une base de Schauder relative à un filtre borélien.
\end{abstract}

\maketitle


\section{Introduction}
\subsection{Order bases in Banach lattices}
Recall that a Banach lattice is a Banach space $X$ equipped with an order relation $\leqslant$  satisfying:
\begin{enumerate}
    \item $\forall x,y,z\in X  \quad x\leqslant y \;\Rightarrow\; x+z\leqslant y+z$.
    \item $\forall x,y\in X\;  \forall \lambda\in\mathbb R_+\quad x\leqslant y \;\Rightarrow \; \lambda x\leqslant \lambda y$.
    \item Any two elements $x,y\in X$ have a least upper bound $x\vee y$ and a greatest lower bound $x\wedge y$.
    \item For any $x\in X$, set $|x|=x\vee(-x)$. Then, if $|x|\leqslant |y|$, also $\|x\|\leqslant \|y\|$.
\end{enumerate}
Let $X$ be a Banach lattice. Then the lattice structure on $X$ gives rise to three classical notions of sequential convergence, not available in a general Banach space. Namely,
\begin{itemize}
\item a sequence $(x_n)_{n=1}^\infty$ {\em converges uniformly} to $x$, denoted $x_n\maps{\sf u}x$, if there is some $z\in X_+$  so that
$$
\a m \;\a^\infty n\; |x_n-x|\leqslant \tfrac zm,
$$
\item a sequence $(x_n)_{n=1}^\infty$ {\em $\sigma$-order converges} to $x$, denoted $x_n\maps{\sf \sigma \sf o}x$, if there is some sequence $z_m\downarrow 0$ in $X$ so that
$$
\a m \;\a^\infty n\; |x_n-x|\leqslant z_m,
$$
\item a sequence $(x_n)_{n=1}^\infty$ {\em order converges} to $x$, denoted $x_n\maps{\sf o}x$, if there is some net $z_{\mu}\downarrow 0$ in $X$ so that
$$
\a  \mu\;\a^\infty n\; |x_n-x|\leqslant z_\mu.
$$
\end{itemize}
Here the notation $\a^\infty n$ means for all but finitely many $n$, i.e., $\e N\;\a n\geqslant N$, while $z_m\downarrow 0$ and $z_{\mu}\downarrow 0$ mean that $(z_m)$ and $(z_\mu)$ are decreasing and have infimum $0$ in $X$. It can be shown that, in all cases above, the limit is unique whenever it exists. Thus, if $C$ is one of the above notions of convergence and  $\sum_{n=1}^\infty x_n$ is a series in $X$, we can unambiguously write
$$
x=^C\sum_{n=1}^\infty x_n
$$
to denote that the sequence of partial sums $(\sum_{m=1}^n x_m)$ $C$-converges to $x$.

 All three notions of convergence are evidently compatible with the algebraic structure of $X$, in the sense that if the sequence of sums of $(x_n)$ and $(y_n)$ converge to $x$ and $y$ respectively, then the sequence of sums of $(x_n+y_n)$ converges to $x+y$ and similarly for scalar products. Nevertheless, neither uniform nor order convergence arise in general from Hausdorff topologies on $X$ (see \cite[Section 18]{Tay} for nets).

As is evident from the definitions, uniform convergence implies norm convergence and $\sigma$-order convergence, whereas $\sigma$-order convergence implies order convergence. However, in the absence of other hypotheses on $X$, no other implications hold (see \cite[Examples 1.2 and 1.3]{taylor} and \cite[Example 18.5]{Tay}), which is recorded in Figure \ref{diagram1}.

\begin{figure}[!htb]\label{diagram1}
\begin{tikzcd}
{\sf uniform} \arrow[-{Implies},double]{d}\arrow[-{Implies},double]{r}& {\sf norm} \arrow[-{Implies},double,negated]{d}\\
 \sigma-{\text{}\sf order}\arrow[-{Implies},double,negated]{ur} \arrow[-{Implies},double]{r}{}  &{\text{}\sf order}\arrow[-{Implies},double,negated,bend right=-30]{l}{} &\\
\end{tikzcd}
\caption{Implications between convergence types.}
\end{figure}
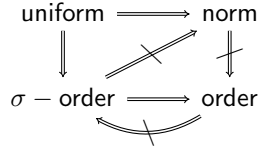
We now define bases associated with the above convergences.
\begin{defi}\label{C-basis}
Let $X$ be a Banach lattice and  $C$ one of the following convergence types: norm, uniform, order or $\sigma$-order.  Then a sequence $(e_n)_{n=1}^\infty$ in $X$ is said to be a {\em $C$-basis} for $X$ provided that, for every $x\in X$, there is a unique sequence of scalars $(a_n)\in \R^\N$ so that  
$$
x=^C\sum_{n=1}^\infty a_ne_n.
$$
\end{defi}
For example, it is easy to see that the standard unit vector sequence $(e_n)$ forms a $\sigma$-order basis for each of the spaces $\ell_p$, $1\leqslant p\leqslant \infty$, and also for $c_0$. We are thus in the unfamiliar situation that the same sequence $(e_n)$ is a $\sigma$-order basis both for $\ell_\infty$ and its subspace (even sublattice) $c_0$. In particular, we see that the norm-closed linear span $[e_n]$ of $(e_n)$ may be strictly smaller than the lattice $X$ for which it is a $\sigma$-order basis.

When $(e_n)$ is a $C$-basis for $X$, we may define functionals $e^\sharp_k:X\maps{}\R$ by letting
$$
{e^\sharp_k}(x)=a_k,
$$
where $(a_n)$ is the uniquely defined sequence referenced above. Similarly, we let $P_m:X\maps{} [e_1,\ldots, e_m]$ denote the corresponding sequence of basis projections,
$$
P_m(x)=\sum_{n=1}^me^\sharp_n(x)e_n.
$$
Since the sequence constantly equal to $e_n$ will $C$-converge to $e_n$, we find that  ${e^\sharp_k}(e_n)=\delta_{k,n}$ for all $k,n$, that is, the functionals $e^\sharp_k$ are {\em biorthogonal} to the sequence  $(e_n)$.
Observe however that a priori it is not clear that the functionals $e^\sharp_k$ or the operators $P_m$ are continuous (with respect to the norm topology on $X$; Banach lattices often admit no order continuous linear functionals).

Norm bases are of course more commonly known as {\em Schauder bases} and we shall employ that terminology here. Moreover, it is a classical result \cite[p.~111]{Banach} that the biorthogonal functionals $e^\sharp_k$ associated  to a Schauder basis are always continuous. 

Biorthogonal functionals associated with some sequence  are typically denoted by $e^*_k$, but, since the very continuity of the functionals is at play here, we shall only use the notation $e^*_k$ if we already know that they are continuous.

In our first main theorem, which settles the relationships between the different types of bases, for one of the implications we resort to an additional set-theoretical axiom, namely the determinacy of certain infinite games on $\N$ \cite[Definition 26.3]{Kechris}. Nevertheless, this usage should not be too disturbing as analytic determinacy is arguably part of the right set-theoretical foundations of mathematics. In any case, what we actually use in our argument is the weaker assumption that ${\bf \Sigma}^1_2$ sets are Baire measurable.

\begin{thm}[Analytic determinacy]\label{super}
Suppose $(e_n)$ is a sequence of vectors in a Banach lattice $X=[e_n]$ and $(e^\sharp_n)$ is a sequence of (possibly discontinuous) biorthogonal functionals for $(e_n)$. Consider the following properties:
\begin{enumerate}
\item\label{order basis} $(e_n)$ is an order basis for $X$ with corresponding functionals $(e^\sharp_n)$,
\item\label{sigma-order basis} $(e_n)$ is a $\sigma$-order basis for $X$  with corresponding functionals $(e^\sharp_n)$,
\item\label{uniform basis} $(e_n)$ is a uniform basis for $X$  with corresponding functionals $(e^\sharp_n)$,
\item\label{schauder} $(e_n)$ is a Schauder basis for $X$ with corresponding functionals $(e^\sharp_n)$.
\end{enumerate}
Then  (1)$\Leftrightarrow$(2)$\saa$(3)$\saa$(4). In particular, the $e_n^\sharp$ are continuous in all the kinds of bases considered above.
\end{thm}
As hinted above, analytic determinacy is only needed for part of the theorem; namely, the implication (2)$\saa$(3). All other implications hold without any additional set-theoretical assumptions. Examples showing that (4)$\not\saa$(3) and (3)$\not\saa$(2)  can be found in \cite[Example 6.2]{taylor} and \cite[Example 9.3]{taylor}, respectively.

Coupled with \cite[Theorem 2.1]{taylor}, we obtain the following characterisation of uniform bases for Banach lattices.

\begin{cor}\label{cor:uniform}
Let $(e_n)$ be a sequence of non-zero vectors in a Banach lattice $X$ such that $X=[e_n]$. The following conditions are equivalent:
\begin{enumerate}
\item $(e_n)$  is a uniform basis for $X$,
\item $(e_n)$ is a Schauder basis for $X$ so that, for every $x\in X$, the sequence of partial sums
$$
P_mx=\sum_{n=1}^me^*_n(x)e_n
$$
is order bounded,
\item\label{bibasic ineq} there is a constant $M$ so that, for all finite tuples of scalars $(a_n)_{n=1}^m$ one has 
\maths{
\BNORM{    \bigvee_{k=1}^m\Big|\sum_{n=1}^ka_ne_n\Big|  \,  }\leqslant M\Bigg\|\sum_{n=1}^m a_ne_n\Bigg\|.
}
\end{enumerate}
\end{cor}

\begin{rem}
    It is a classical fact (see \cite[Remark 1.3]{Tay}) that, if $X$ is a Banach lattice of measurable functions (i.e.~an ideal in the space of measurable functions $L_0(\Omega,\Sigma,\mu)$ for some semi-finite measure space $(\Omega,\Sigma,\mu)$), we have $f_n\conv{\sf o}{n}f$ if and only if $f_n\conv{a.e.}{n}f$ and, moreover,  there exists a $g\in X_+$ satisfying  $|f_n|\leqslant g$ for all $n$ (i.e., the sequence $(f_n)$ is order bounded).  Therefore, a uniform basis $(e_n)$ of a Banach lattice $X$ can be thought of as a coordinate system which guarantees both norm and dominated almost everywhere convergence of the basis expansions. 
    
    Statement (3) of Corollary~\ref{cor:uniform}  is simply the standard inequality for Schauder basic sequences with the supremum $\bigvee_{k=1}^m$ pulled inside the norm. The equivalence between (1) and (3) in Corollary~\ref{cor:uniform} therefore shows that bounding the maximal function of a basic sequence is equivalent to establishing strong convergence properties of the series, even in the general setting of Banach lattices. 

We remark that  ${\sf u}$-bases (and ${\sf u}$-basic sequences; see Remark~\ref{rem on Cor13}) occur frequently in applications. For example, it follows from Doob's inequality  that martingale difference sequences in $L_p(\mu)$ ($1<p<\infty$ and $\mu$ a probability measure) are ${\sf u}$-basic (see \cite[Example 2.11]{taylor}). On the other hand, the combination of the Burkholder--Davis--Gundy and Khintchine inequalities yields that unconditional blocks of the Haar basis in $L_1[0,1]$ are ${\sf u}$-basic (see \cite[Proposition 8.2]{taylor}), and the Carleson--Hunt theorem  \cite{Hunt} establishes the inequality in Corollary~\ref{cor:uniform} (3) for the trigonometric basis. For several more examples and non-examples of ${\sf u}$-basic sequences, the reader may consult \cites{Tay, taylor}. 

\end{rem}

Analogously to Corollary~\ref{cor:uniform},  we may characterise $\sigma$-order bases as follows.

\begin{cor}[Analytic determinacy]\label{cor:order}
The following conditions are equivalent for a  sequence $(e_n)$ of non-zero vectors in a Banach lattice $X=[e_n]$:
\begin{enumerate}
\item $(e_n)$  is a $\sigma$-order basis for $X$,
\item $(e_n)$ is a uniform  basis for $X$ such that $0=^{\sigma\sf o}\sum_{n=1}^\infty 0 e_n$ is the unique $\sigma$-order expansion of $0$.
\end{enumerate}
\end{cor}

Theorem \ref{super} immediately solves several problems listed in the literature regarding the relationships between these basis notions. 
\begin{prob}\label{Question 2.10}\cite[Question 2.10]{taylor}
If $(e_n)$ is simultaneously a Schauder basis and a $\sigma$-order basis for a Banach lattice $X$, do the coefficients in the norm and in the order expansions of the same vector agree? That is, are the two associated sets of biorthogonal functionals equal?
\end{prob}
By the implications (2)$\saa$(3)$\saa$(4), proved under the assumption of analytic determinacy, the answer to Problem \ref{Question 2.10} is therefore positive. 

Recall from \cite{taylor} that a \emph{bibasis} is, by definition, a sequence $(e_n)$ in a Banach lattice $X$ which is simultaneously a Schauder and a uniform basis for $X$. 
\begin{prob}\label{Question 9.2}\cite[Question 9.2]{taylor}
Suppose $(e_n)$ is a bibasis with unique $\sigma$-order expansions. Does $(e_n)$ have unique order expansions?
\end{prob}
By the equivalence of statements (1) and (2) in Theorem~\ref{super}, the answer to Problem~\ref{Question 9.2} is positive.

For our next applications, we need to recall some facts about order continuous Banach lattices from \cite{MR0569521}.
\begin{lemme}
Suppose $X$ is an order continuous Banach lattice. Then, for all sequences $(x_n)$ and vectors $x$,
$$
x_n\maps{\sf o}x\quad\equi\quad x_n\maps{\sigma{\sf  o}}x\quad\equi\quad x_n\maps{\sf u}x.
$$
Similarly, if $X$ is $\sigma$-order continuous, then
$$
x_n\maps{\sigma{\sf  o}}x\quad\equi\quad x_n\maps{\sf u}x.
$$
\end{lemme} 
Thus, in an order continuous Banach lattice $X$, the notions of order, $\sigma$-order and uniform bases coincide and they will have the same associated biorthogonal functionals. Similarly, if $X$ is $\sigma$-order continuous, the notions of $\sigma$-order and uniform bases coincide and have the same biorthogonal functionals. In particular, these types of bases will automatically also be Schauder bases with the same biorthogonal functionals. 

\begin{cor}
Let $(e_n)$ be a $\sigma$-order basis for a $\sigma$-order continuous Banach lattice. Then $(e_n)$ is also a Schauder basis with the same biorthogonal functionals.
\end{cor}
This in turn provides a positive answer to the following questions.

\begin{prob}\cite[Problem 1.3]{gumenchuk} Let $(x_n)$ be a $\sigma$-order basis of an order continuous Banach lattice E. Is then $(x_n)$ a Schauder basis of $E$? What about $E = L_p$ with $1\leqslant p<\infty$?
\end{prob}

In the same paper, the authors consider the specific example of $L_1$ and conjecture a negative answer to the following question.
\begin{prob}\label{Problem 5.2}\cite[Problem 5.2]{gumenchuk} 
Does $L_1$ have a $\sigma$-order basis?
\end{prob}
Regarding this, they show that $L_1$ does not admit a sequence $(e_n)$ that is simultaneously a Schauder and a $\sigma$-order basis for $L_1$ \cite[Theorem 5.1]{gumenchuk}. However, given that $L_1$ is order continuous, every $\sigma$-order basis for $L_1$ is also a Schauder basis, which gives a negative answer to Problem \ref{Problem 5.2}.
\begin{cor}
The Banach lattice $L_1$ admits no $\sigma$-order basis.
\end{cor}

A few cautionary  remarks on the terminology are in order. Namely, our notion of $\sigma$-order convergence for sequences is simply called {\em order convergence} for sequences in \cite{gumenchuk} and therefore our notion of $\sigma$-order basis is similarly designated {\em order basis} in \cite{gumenchuk}. In the same paper, a sequence $(e_n)$ in a Banach lattice $X$ which is simultaneously a Schauder basis and a $\sigma$-order basis for $X$ is denoted a {\em bibasis}. By Theorem \ref{super}, under analytic determinacy, bibases are thus simply $\sigma$-order bases such that $X=[e_n]$. However, in \cite{gumenchuk}, the authors exclusively work in $\sigma$-order continuous Banach lattices, where of course the notions of $\sigma$-order and uniform bases coincide. On the other hand, in \cite{taylor}, a sequence $(e_n)$ in a Banach lattice $X$ is called a {\em bibasis} for $X$ provided that it is both a Schauder basis and a uniform basis. Because of Theorem \ref{super}, these two competing notions of bibases are superfluous as they just correspond to $\sigma$-order and uniform bases respectively. To avoid any confusion, we shall exclusively employ the terminology of Definition \ref{C-basis} and eschew the, in hindsight, unnecessary notion of bibasis.

It is of course natural to ask whether the use of analytic determinacy in Theorem \ref{super} is really necessary, i.e., if there is not some other more insightful proof bypassing these issues. We do not know the answer to this question, but note that the problem resides in the fact that the very notion of $\sigma$-order convergence is a priori of too high descriptive complexity. Indeed, in Proposition~\ref{prop conv agree}, we are only able to show that, for a general separable Banach lattice $X$, the set
\maths{
\Big\{ \big((x_n),x\big)\in X^\N\times X\Del x_n\conv{\sigma\sf o}{n}x\Big\}
}
is  ${\bf \Delta}^1_2$. On the other hand, if we can show this set to be ${\bf \Sigma}^1_1$ (i.e., analytic) instead, then our proofs no longer necessitate the additional set-theoretical assumptions.  This happens, for example, if the Banach lattice $X$ admits a countable {\em $\pi$-basis}, i.e., a countable set $P\subseteq X$ of positive elements so that, for every $x>0$, there is some $p\in P$ with $0<p\leqslant x$. We will provide elsewhere a hierarchy, indexed by countable ordinal numbers, which can be used to characterize when the above set is analytic.

\subsection{Filter bases}
The second topic of our study concerns a generalisation of Schauder bases in the context of general Banach spaces, not lattices.  Assume that $\mathcal{F}$ is a filter of subsets of $\N$, that is, $\mathcal{F}\subseteq \ku P(\N)$ is closed under taking intersections and supersets,
\begin{itemize}
\item $a,b\in \mathcal{F}\;\saa\; a\cap b\in \mathcal{F}$,
\item $a\subseteq b \;\&\; a\in \mathcal{F}\;\saa\; b\in \mathcal{F}$.
\end{itemize}
For reasons that will become apparent later, we shall also assume that all filters are {\em proper}, i.e.,  $\tom \notin \mathcal{F}$, and contain the {\em Fréchet filter} consisting of all cofinite subsets of $\N$. Recall that a sequence $(x_n)$ is said to {\em converge along $\mathcal{F}$} to $x$, denoted $x_n\conv{}{n\to \mathcal{F}}x$, if
$$
\big\{n\in \N\del\, \norm{x_n-x}<\eps\big\}\in \mathcal{F}
$$
for all $\eps>0$.
In complete analogy with Definition \ref{C-basis}, we have the following definition due to M. Ganichev and V. Kadets \cite{ganichev}.
\begin{defi}
A sequence $(e_n)$ in a Banach space $X$ is said to be an {\em $\mathcal{F}$-basis for $X$} provided that, for all $x\in X$, there is a unique sequence $(a_n)\in \R^\N$ so that $\sum_{n=1}^ma_ne_n\conv{}{m\to \mathcal{F}}x$, which we denote by
$$
x=^\mathcal{F}\sum_{n=1}^\infty a_ne_n.
$$
More generally, $(e_n)$ is said to be a {\em filter basis for $X$} if it is an $\mathcal{F}$-basis for $X$ for some filter $\mathcal{F}$, in which case $\mathcal{F}$ is said to be {\em compatible} with $(e_n)$.
\end{defi}
Let us note that, if $\mathcal{F}$ is just the Fréchet filter itself, then an $\mathcal{F}$-basis $(e_n)$ is nothing but a  Schauder basis.
Although \cite[Example 1]{kadets} provides a basis for $\ell_2$ with respect to the filter of sets of density $1$, which however is not a Schauder basis, T. Kania asked whether there is an example of a Banach space without a Schauder basis that nevertheless has an $\mathcal{F}$-basis for some appropriate filter $\mathcal{F}$. We answer this by the following simple example.
\begin{exa}[A Banach space with a filter basis, but no Schauder basis]
Let $X$ be a Banach space with a finite-dimensional decomposition $(X_n)_{n=1}^\infty$, but without a  Schauder basis. That such spaces exist follows for example from \cite[Theorem 1.1]{szarek}. Choose now sequences $(e_i)_{i\in \N}$ and $0=k_0<k_1<k_2<\ldots$ so that
$$
\big\{e_i \del   k_{n-1}< i \leqslant  {k_n}\big\}
$$
is a basis for $X_n$ for all $n$. Let also 
$$
\mathcal{F}=\big\{ a\subseteq \N \del k_n\in a \text{ for all but finitely many }n\big\}.
$$
Since $(X_n)$ is an F.D.D. for $X$, we have that, for every $x\in X$, there are unique vectors $x_n\in X_n$ so that $x=\sum_{n=1}^\infty x_n$. Writing $x_n=\sum_{i=k_{n-1}+1}^{k_n}a_ie_i$ for appropriate scalars $a_i\in \R$, we see that, for all $\eps>0$, 
$$
k_n\in \Big\{m\in \N\Del\, \NORM{x-\sum_{i=1}^ma_ie_i}<\eps\Big\}
$$
for all but finitely many $n$ and hence that the latter set belongs to $\mathcal{F}$. Furthermore, by the uniqueness of the $x_n$, $(a_i)$ is the only such sequence, which shows that $(e_i)$ is an $\mathcal{F}$-basis for $X$.
\end{exa}

Note that, if $\mathcal{F}$ is a filter and $(e_n)$ is an $\mathcal{F}$-basis for $X$, we may define the associated biorthogonal functionals $e^\sharp_n$ just as for order bases etc. However, it might be possible that $(e_n)$ is simultaneously an $\mathcal{F}'$-basis for $X$ with respect to some other filter $\mathcal{F}'$, in which case it is unclear whether the biorthogonal functionals associated with $\mathcal{F}'$ are the same as those associated with $\mathcal{F}$. Furthermore, even under additional set-theoretical axioms, it is no longer clear whether any of these functionals are continuous. To discuss these issues, we must introduce a more refined concept. 
\begin{defi}
Let $(e_n,e^\sharp_n)$ be a biorthogonal system in a Banach space $X$, i.e., $(e_n)$ is a sequence of vectors in $X$  and $e^\sharp_n\colon X\to \R$ are (possibly discontinuous) functionals biorthogonal to the $e_n$. We say that $(e_n,e^\sharp_n)$  is  a {\em filter basis system} for  $X$ provided that there is a filter $\mathcal{F}$ so that $(e_n)$ is an $\mathcal{F}$-basis for $X$ with associated biorthogonal functionals $e^\sharp_n$. Such a filter $\mathcal{F}$ is said to be {\em compatible} with $(e_n,e^\sharp_n)$.
\end{defi}
Let us note that, to every filter basis system $(e_n,e^\sharp_n)$, there is a smallest compatible filter $\mathcal{F}$, which we shall return to later on. Recall also that a sequence $(e_n)$ in a Banach space is said to be {\em minimal} if, for all $k$, we have that $e_k\notin [e_n]_{n\neq k}$. 
\begin{lemme}\label{minimal-cont}
Let $(e_n,e^\sharp_n)$ be a filter basis system for $X$. Then the sequence $(e_n)$ is minimal if and only if the functionals $e^\sharp_n$ are continuous.
\end{lemme}
Thus, the continuity of the associated biorthogonal functionals can be detected directly on the filter basis $(e_n)$ itself without even involving the functionals.

Although Ganichev and Kadets \cite{ganichev} operate with a slightly more general notion of filter basis, the following problem remains open even in our setting.
\begin{prob}\label{ganichev}\cite{ganichev}
Suppose $(e_n)$  is a filter basis for a Banach space $X$. Is $(e_n)$ necessarily a minimal sequence?
\end{prob}

The papers \cites{kania, Rancourt} address Problem \ref{ganichev} and show that a  filter basis $(e_n)$ is minimal if and only if it admits a compatible filter $\mathcal{F}$ that is analytic when viewed as a subset of $\ku P(\N)=\{0,1\}^\N$ \cite[Theorem A, Theorem B]{Rancourt}.
In connection with this, \cite[Question 2]{Rancourt} asks whether it is possible to improve this so as to get the filter $\mathcal{F}$ to be Borel.  We resolve this even while  keeping the associated biorthogonal functionals $e^\sharp_n$ fixed.

\begin{thm}\label{thm:filter}
Let $(e_n,e^\sharp_n)$ be a filter basis system for $X$. Then the following are equivalent.
\begin{enumerate}
\item The functionals $e^\sharp_n$ are continuous,
\item the sequence $(e_n)$ is minimal, 
\item the smallest compatible filter is analytic,
\item there is a compatible analytic filter,
\item there is a compatible Borel filter.
\end{enumerate}
\end{thm}

Note that examples of analytic non-Borel filters are well-known (see for example \cite[Proposition 6.3]{Uzcategui}).



\section{Proofs for order bases}\label{Section o}
As is well-known from the case of Banach spaces, it is often useful to operate with basic sequences as opposed to bases. So let us introduce this notion in our context. Recall first that a sequence $(e_n)$ in a Banach space $X$ is said to be {\em Schauder basic} in case it is a Schauder basis for its closed linear span $[e_n]$. When dealing with  uniform, $\sigma$-order and order convergence, extra caution is required since the closed linear span $[e_n]$ of a sequence in $X$  need not be a sublattice. Furthermore, even the notions of $\sigma$-order and order convergence are not absolute, but depend on the ambient lattice (for example the unit vector basis $(e_k)$ is not order convergent in $c_0$, but it is order null when viewed in $\ell_\infty$). On the other hand, uniform convergence is absolute \cite[Proposition 2.12]{taylor}.  In fact, as shown in Lemma \ref{lem:uniform} below, uniform convergence can be equivalently reformulated so as to avoid any reference to the ambient lattice.  

\begin{lemme}\label{lem:uniform}
For a sequence $(x_n)$ and vector $x$ in a Banach lattice $X$, we have that
$x_n\conv{\sf u}{n}x$ if and only if
$$
\a \eps>0\; \e k\; \a m\geqslant k\quad \BNORM{\bigvee_{n=k}^m|x_n-x|}<\eps.
$$
In particular, the set
$$
\Big\{ \big((x_n)_{n=1}^\infty,x\big)\in X^\N\times X\Del x_n\conv{\sf u}{n}x\Big\}
$$ 
is Borel.
\end{lemme}

\begin{proof}
Suppose first that $x_n\conv{\sf u}{n}x$ and find some $z>0$ so that
$$
\a l\;\a^\infty n\;|x_n-x|\leqslant \tfrac zl.
$$
Thus, if $\eps>0$ is given, choose $l$ large enough that $\norm{\tfrac zl}<\eps$ and find $k$ so that $|x_n-x|\leqslant \tfrac zl$ for all $n\geqslant k$. We then see that
$$
\BNORM{   \bigvee_{n=k}^m|x_n-x|}\leqslant \norm{\tfrac zl}<\eps
$$
for all $m\geqslant k$.

Conversely, suppose that 
$$
\a \eps>0\; \e k\; \a m\geqslant k\quad \BNORM{\bigvee_{n=k}^m|x_n-x|}<\eps.
$$
We choose $0=k_0<k_1<k_2<\ldots$ so that
$$
\BNORM{\bigvee_{n=k_l}^m|x_n-x|}<\tfrac 1{4^l}
$$
whenever $l\geqslant 1$ and $m\geqslant k_l$. This implies that the series
$$
\sum_{l=0}^\infty 2^l\bigvee_{n=k_l}^{k_{l+1}-1}|x_n-x|
$$
converges in norm to some element $z\in X_+$. Moreover, for any $l\geqslant 1$ and all $n\geqslant k_l$, we have that $2^l|x_n-x|\leqslant  z$ and hence $|x_n-x|\leqslant  \tfrac z{2^l}$. Therefore, $x_n\conv{\sf u}{n}x$.
\end{proof}

We may thus define a sequence $(e_n)$ in a Banach lattice $X$ to be {\em ${\sf u}$-basic} in case, for every $x\in [e_n]$, there is a unique sequence $(a_n)\in \R^\N$ so that $x=^{\sf u}\sum_{n=1}^\infty a_ne_n$. By Lemma \ref{lem:uniform}, this notion is intrinsically defined and independent of the choice of the ambient Banach lattice, which we may therefore always assume to be the separable Banach lattice $[e_n]_\wedge\subseteq X$ generated by $(e_n)$. Note however that $[e_n]$ itself will in general only be a Banach space, not a lattice. The corresponding biorthogonal functionals $e^\sharp_k\colon [e_n]\to \R$ are defined as before.

The following establishes the implication (3)$\saa$(4) of Theorem \ref{super}. 
\begin{thm}\label{master1}
Suppose that $(e_n)$ is a $\sf u$-basic sequence in a Banach lattice $X$. Then the biorthogonal functionals $e^\sharp_n$ are continuous and hence $(e_n)$ is Schauder basic.
\end{thm}

To simplify notation, it is slightly easier to work with the operator $E:[e_n]\maps{}\R^\N$ defined by $Ex=\big(e_n^\sharp(x)\big)_n$ in place of the biorthogonal functionals themselves. Observe that $E$ is continuous if and only if all the $e^\sharp_n$ are continuous. 

\begin{proof}
We recall that, since both $[e_n]$ and $\R^\N$ are separable Fréchet spaces and hence Polish spaces, the operator $E$ is continuous if and only if the graph $\ku G E\subseteq [e_n]\times \R^\N$ is Borel. Indeed, if $\ku G E$ is Borel or even analytic, then $E$ is Borel measurable \cite[Theorem 14.12]{Kechris} and therefore continuous \cite[Theorem 9.10]{Kechris}. Now,  
\maths{
\big(x,(a_n)\big)\in \ku GE
&\;\equi\;    \sum_{n=1}^m a_ne_n\conv{\sf u}{m}x,
}
which is Borel by Lemma \ref{lem:uniform}. Thus, the biorthogonal functionals $e^\sharp_n$ are all continuous. 

To see that $(e_n)$ is Schauder basic, i.e., a Schauder basis for $[e_n]$, note that, for every $x\in [e_n]$, we have that 
$ x=^{\sf u}\sum_{n=1}^\infty e^\sharp_n(x)e_n$ and hence also $ x=^{\norm{\cdot}}\sum_{n=1}^\infty e^\sharp_n(x)e_n$. Also, if $(a_n)\in \R^\N$ is any sequence so that $x=^{\norm{\cdot}}\sum_{n=1}^\infty a_ne_n$, we see that
$$
a_k
=\lim_m e^\sharp_k\Big(\sum_{n=1}^m a_ne_n\Big)
=e^\sharp_k\Big(\lim_m\sum_{n=1}^m a_ne_n\Big)
=e^\sharp_k(x),
$$
so the norm-expansion $ x=^{\norm{\cdot}}\sum_{n=1}^\infty e^\sharp_n(x)e_n$ is unique.
\end{proof}
In order to obtain a similar result for order bases, we must first reformulate the definition of order convergence as we did with uniform convergence in Lemma~\ref{lem:uniform}.
\begin{lemme}\label{lem:order bases}
For a sequence $(x_n)$ and vector $x$ in a Banach lattice $X$, we have that
$x_n\conv{\sf o}{n}x$ if and only if
\begin{equation*}
    \a y>0\;\e z\; \big(y\not\leqslant z \;\;\&\;\; \a^\infty n\; |x_n-x|\leqslant z\big).
\end{equation*} 
\end{lemme}
\begin{proof}
To establish our claim, suppose first   that $x_n\conv{\sf o}{}x$. This means that there is a decreasing net $(z_\mu)$ with $\inf_\mu z_\mu=0$ so that $\a \mu\;\a^\infty n\;   |x_n-x|\leqslant z_\mu$.
In particular, if $y>0$, then $y\not\leqslant z_\mu$ for some $\mu$ and hence the sequence $(x_n)$ satisfies 
\begin{equation}\label{order}
 \a y>0\;\e z\; \big(y\not\leqslant z \;\;\&\;\; \a^\infty n\; |x_n-x|\leqslant z\big).
\end{equation}
Conversely, suppose that $(x_n)$ satisfies (\ref{order}). Then the set
$$
A=\{z\in X\del \a^\infty n\; |x_n-x|\leqslant z\}
$$
becomes directed under the ordering $z\prec z'\equi z'\leqslant z$. For if $z,z' \in A$, then also $z,z'\prec z\wedge z'\in A$. It follows that $(A,\prec)$ can be viewed as a decreasing net with infimum $0$ witnessing that $x_n\conv{\sf o}{}x$.
\end{proof}

Using Lemma~\ref{lem:order bases}, we may now show that order and $\sigma$-order convergence agree in separable Banach lattices.
\begin{prop}\label{prop conv agree}
Let $X$ be a separable Banach lattice. Then a sequence $(x_n)$ in $X$ order converges to $x\in X$ if and only if it $\sigma$-order converges to $x$.  In particular, the set 
\begin{equation}\label{compl of sigma-o}
    \Big\{ \big((x_n)_{n=1}^\infty,x\big)\in X^\N\times X\Del x_n\conv{\sigma\sf o}{n}x\Big\}
\end{equation}
is ${\bf \Delta}^1_2$.
\end{prop}
\begin{proof}
By Lemma~\ref{lem:order bases}, if the sequence $(x_n)$ order converges to $x$, we have that $\inf(A)=0,$ where $A$ is the set of all eventual upper bounds of $(|x_n-x|)$, i.e., $$A=\{z\in X\del \a^\infty n\; |x_n-x|\leqslant z\}.$$ By separability, there is a countable norm dense subset $D=\{d_1,d_2,\dots\}$ of $A$. Let $z_m=\inf\{d_1,\dots,d_m\}$. Then $(z_m)$ is a decreasing sequence of eventual upper bounds of $\big(|x_n-x|\big)$, so it is enough to check that it has infimum $0$. Suppose not. Then there exists  a positive lower bound $0<y\leqslant z_m$. In particular, we  have that $y\leqslant d_m$ for all $m$. By density, this  implies that $y\leqslant z$ for all $z\in A$, which contradicts that the infimum of $A$ is $0$.

Finally, note that by the reformulation in Lemma~\ref{lem:order bases}, order convergence is clearly a  ${\bf \Pi}^1_2$-condition on $\big((x_n),x\big)\in X^\N\times X$. On the other hand, using the definition of $\sigma$-order convergence, for $x\in X$ and $(x_n)\in X^\N$, we have 
\maths{
x_n\conv{\sigma\sf o}{n}x
&\;\equi\;    \e (z_m)\in X_+^\N \quad\Big( z_m\downarrow0 \quad  \&\quad   \a m\; \a^\infty n \;\;\; \Big|  x -  x_n\Big|\leqslant z_m\Big).
}
Clearly, the condition $\a m\; \a^\infty n \; \Big|  x -  x_n\Big|\leqslant z_m$ is Borel in the tuple $\big(x,(x_n),(z_m)\big)\in X\times X^\N\times X_+^\N$, but unfortunately the condition $z_m\downarrow0$ appears only to be ${\bf \Pi}^1_1$ in $(z_m)\in X_+^\N$:
$$
z_m\downarrow0 \;\;\equi\;\; \a m\; z_m\geqslant z_{m+1} \quad\&\quad \a y>0 \; \e m \; y\not<z_m.
$$
Hence, it follows that the set \eqref{compl of sigma-o} is also ${\bf \Sigma}^1_2$. Consequently,  it  is ${\bf \Delta}^1_2$.
\end{proof}

Due to the higher complexity of $\sigma$-order convergence, in order to  obtain a result analogous to Theorem~\ref{master1} for   $\sigma$-order bases, we are forced to rely on additional set-theoretical assumptions, namely, the determinacy of increasingly complicated sets. For the explicit description of ${\bf \Sigma}^1_1$-determinacy we refer the reader to \cite[(26.3)]{Kechris}, and Martin's Axiom (MA) to \cite[Chapter 19]{Just}.

\begin{thm}[${\bf \Sigma}^1_1$-determinacy or $\text{MA}+\neg \text{CH}$]\label{master2}
Suppose that $(e_n)$ is a $\sigma$-order basis for a separable Banach lattice $X$. Then the biorthogonal functionals $e^\sharp_n$ are continuous and hence $(e_n)$ is Schauder basic.
\end{thm}

\begin{proof}
Observe that, for $x\in X$ and $(a_n)\in \R^\N$, we have 
\maths{
x=^{\sigma\sf o }\sum_{n=1}^\infty a_ne_n
&\;\equi\;    \e (z_m)\in X_+^\N \quad\Big( z_m\downarrow0 \quad  \&\quad   \a m\; \a^\infty n \;\;\; \Big|  x -  \sum_{k=1}^na_ke_k\Big|\leqslant z_m\Big).
}
As in the proof of Proposition \ref{prop conv agree}, the condition $\a m\; \a^\infty n \; \Big|  x -  \sum_{k=1}^na_ke_k\Big|\leqslant z_m$ is Borel in the tuple $\big(x,(a_n),(z_m)\big)\in X\times \R^\N\times X_+^\N$, but unfortunately the condition $z_m\downarrow0$ appears only to be ${\bf \Pi}^1_1$ in $(z_m)\in X_+^\N$,
$$
z_m\downarrow0 \;\;\equi\;\; \a m\; z_m\geqslant z_{m+1} \quad\&\quad \a y>0 \; \e m \; y\not<z_m.
$$
Thus, a priori, the graph $\ku GE$  of $E:X\maps{}\R^\N$ is only ${\bf \Sigma}^1_2$, which means that the inverse image $E\inv(U)$ of an open set $U\subseteq \R^\N$ is ${\bf \Sigma}^1_2$ and therefore has the property of Baire if we assume either ${\bf \Sigma}^1_1$-determinacy \cite[Theorem 36.20]{Kechris} or $\text{MA}+\neg \text{CH}$ \cite[Exercise 38.8]{Kechris}, \cite[Theorem 19.23]{Just}, \cite[Theorem 3.15]{Cichonetal}. Therefore, $E$ is continuous by \cite[Theorem 9.10]{Kechris} and hence so are the associated partial sum projections $P_m:X\maps {}[e_1,\ldots, e_m]$.

We claim that the sequence of operators $(P_m)_{m=1}^\infty$ is uniformly bounded, which by Grunblum's criterion \cite[Proposition 1.1.9]{albiac} implies that $(e_n)$ is Schauder basic. To see this, it suffices by the principle of uniform boundedness to show that  $(P_mx)_{m=1}^\infty$ is bounded in norm for each $x\in X$. However, given $x$, observe that, as $P_mx\underset{m\to \infty}{\overset{\sigma\sf o}\longrightarrow} x$, there is $z\geqslant 0$ for which $|x-P_mx|\leqslant z$ for all but finitely many $m$, which shows that the sequence $\big(\norm{P_mx}\big)_{m=1}^\infty$ is bounded.
\end{proof}
The proofs of Theorem \ref{super} and Corollaries \ref{cor:uniform} and \ref{cor:order} heavily rely on the relationship between the different types of convergence of partial sums established in \cite[Theorem 2.1]{taylor} (see also \cite[Theorem 2.3]{gumenchuk} for a related earlier result).
\begin{thm}\cite[Theorem 2.1]{taylor}\label{bibasis theorem}
The following statements are equivalent for a Schauder basic sequence $(e_n)$ in a Banach lattice $X$ with associated basis projections $P_m:[e_n]\maps{}[e_n]$. 
\begin{enumerate}
\item[(i)] For all $x\in [e_n]$, $P_mx\xrightarrow{\sf u}x$,
\item[(ii)] For all $x\in [e_n]$, $P_mx\xrightarrow{\sigma \sf o}x$,
\item[(iii)] For all $x\in [e_n]$, $P_mx\xrightarrow{\sf o}x$,
\item[(iv)] For all $x\in [e_n]$, $(P_mx)$ is order bounded in $X$,
\item[(v)] For all $x\in [e_n]$, $(\bigvee_{n=1}^m\left|P_nx\right|)$ is norm bounded,
\item[(vi)] There is $M\geq1$ so that, for all $n\in \mathbb{N}$ and scalars $a_1,\ldots, a_n$,  one has 
\maths{
\BNORM{    \bigvee_{m=1}^n\Big|\sum_{k=1}^ma_ke_k\Big|  \,  }\leqslant M\Bigg\|\sum_{k=1}^n a_ke_k\Bigg\|.
}
\end{enumerate}
\end{thm}

\begin{proof}[Proof of Theorem \ref{super}]
The equivalence (1)$\Leftrightarrow$(2) is an immediate consequence of Proposition~\ref{prop conv agree}, so we focus on the implication (2)$\saa$(3). Assume that $(e_n)$ is a $\sigma$-order basis for $X$  with corresponding functionals $(e^\sharp_n)$. Then, by Theorem \ref{master2}, the biorthogonal functionals $(e^\sharp_n)$ are continuous and $(e_n)$ is a Schauder basis for $X=[e_n]$. Furthermore, for all $x\in X$, we have $P_mx\conv{\sigma\sf o}{}x$, which means that $(e_n)$ satisfies condition (ii) of \cite[Theorem 3.1]{taylor} and hence must also satisfy condition (i) of the same theorem, namely that, for every $x\in X$, $P_mx\conv{\sf u}{}x$, i.e., $x=^{\sf u}\sum_{n=1}^\infty e^\sharp_n(x)e_n$. On the other hand, if $(a_n)$ is any sequence so that $x=^{\sf u}\sum_{n=1}^\infty a_ne_n$, then also $x=^{\sigma\sf o}\sum_{n=1}^\infty a_ne_n$, whereby $a_n=e^\sharp_n(x)$ as $(e_n)$ is a $\sigma$-order basis. This shows uniqueness of the uniform expansion and hence implies that $(e_n)$ is a uniform basis for $X$. Finally, the implication (3)$\saa$(4) follows directly from Theorem \ref{master1}.
\end{proof}

\begin{proof}[Proof of Corollary \ref{cor:uniform}]
Fix a sequence $(e_n)$ of non-zero vectors in a Banach lattice $X$ so that $X=[e_n]$. Assume first that condition (3) holds, i.e., that, for some constant $M$ and all finite tuples of scalars $(a_n)_{n=1}^m$ one has 
\maths{
\BNORM{    \bigvee_{k=1}^m\Big|\sum_{n=1}^ka_ne_n\Big|  \,  }\leqslant M\Bigg\|\sum_{n=1}^m a_ne_n\Bigg\|, 
}
whereby also
\maths{
\BNORM{    \sum_{n=1}^ka_ne_n    }\leqslant M\Bigg\|\sum_{n=1}^m a_ne_n\Bigg\|, 
}
for all $k\leqslant m$. Thus, by Grunblum's criterion \cite[Proposition 1.1.9]{albiac}, we see that $(e_n)$ is a Schauder basis for $X$. We let $P_m$ denote the corresponding basis projections and $e^*_n$ the biorthogonal functionals. By the implication (vi)$\saa$(i) of \cite[Theorem 3.1]{taylor}, we find that, for all $x\in X$, $P_mx\conv{\sf u}{}x$, i.e., that $x=^{\sf u}\sum_{n=1}^\infty e^*_n(x)e_n$. On the other hand, to see that this expansion is unique, note that, if $x=^{\sf u}\sum_{n=1}^\infty a_ne_n$ for some sequence $(a_n)$, then also $x=^{\norm\cdot}\sum_{n=1}^\infty a_ne_n$, which in turn implies that $a_n=e^*_n(x)$ for all $n$. This shows that $(e_n)$ is also a uniform basis for $X$ and hence verifies the implication  (3)$\saa$(1).

Now, assume instead that $(e_n)$ is a uniform basis for $X$. Then, by Theorem \ref{master1}, $(e_n)$ is also a Schauder basis for $X$. Let again $P_m$ denote the corresponding basis projections. Then, for all $x\in X$, $P_mx\conv{\sf u}{}x$, which implies that the sequence $(P_mx)$ is order bounded. Thus (1)$\saa$(2).

Finally, the implication (2)$\saa$(3) is a direct consequence of \cite[Theorem 3.1]{taylor}.
\end{proof}

\begin{proof}[Proof of Corollary \ref{cor:order}]
By Theorem \ref{super}, if $(e_n)$ is a $\sigma$-order basis for $X$, it is also a uniform basis for $X$. Note also that, because $(e_n)$ is a $\sigma$-order basis, $0=^{\sigma \sf o}\sum_{n=1}^\infty 0 e_n$ must be the unique $\sigma$-order expansion of $0$. This shows that (1)$\saa$(2).

Conversely, if (2) holds, then,  by Theorem \ref{super},  $(e_n)$ is also a Schauder basis and must satisfy condition (3) of Corollary \ref{cor:uniform}. So, if $e^*_n$ denote the biorthogonal functionals associated to the Schauder basis $(e_n)$, then by \cite[Theorem 3.1]{taylor} we have that $x=^{\sigma\sf o}\sum_{n=1}^\infty e^*_n(x)e_n$ for all $x\in X$. To see that this order expansion of $x$ is unique, note that, if $x=^{\sigma \sf o}\sum_{n=1}^\infty a_ne_n$ for some sequence $(a_n)$, then $0=^{\sigma \sf o}\sum_{n=1}^\infty \big(e^*_n(x)-a_n\big)e_n$ and so $a_n=e^*_n(x)$ by the uniqueness of the $\sigma$-order expansion for $0$. Thus, $(e_n)$ is a $\sigma$-order basis for $X$.
\end{proof}
\begin{rem}\label{rem on Cor13}
   As a consequence of the above discussion, it follows that a sequence $(e_n)$ of non-zero vectors in a Banach lattice $X$ is ${\sf u}$-basic if and only if the inequality in Theorem \ref{bibasis theorem} (vi) holds.  This shows that Corollary~\ref{cor:uniform} still holds for uniform basic sequences (instead of bases) when the assumption $X=[e_n]$ is dropped. Moreover, it yields a significant generalization of Grunblum's criterion \cite[Proposition 1.1.9]{albiac} for Schauder basic sequences. Indeed, if $(e_n)$ is a sequence of non-zero vectors in a Banach space $E$, then we may always view $E$ as contained in the Banach lattice $X=C(B_{E^*})$. In $C(K)$-spaces, it is clear that uniform convergence agrees with norm convergence --  hence the notions of ${\sf u}$-basic and Schauder basic coincide -- and the supremum in Theorem \ref{bibasis theorem} (vi)  commutes with the norm. Therefore, we recover the standard Grunblum criterion  \cite[Proposition 1.1.9]{albiac} in the particular case $X=C(B_{E^*})$.
\end{rem}

We now provide a condition under which the assumption of analytic determinacy may be eliminated from Theorem \ref{super}. Recall that a {\em $\pi$-basis} for a Banach lattice $X$ is a subset $B\subseteq X$ for which  $b>0$ for all $b\in B$ and so that, for all $x>0$, there is $b\in B$ with $b<x$. Observe that, for example, $C([0,1])$ and the sequence spaces $c_0$ and $\ell_p$, $1\leqslant p\leqslant \infty$, all have countable $\pi$-bases, while $L_p[0,1]$ fails to have a countable $\pi$-basis.

\begin{lemme}
Suppose $X$ is a separable Banach lattice with a countable $\pi$-basis $B$. Then, for all sequences $(x_n)$ and vectors $x$, we have
\maths{
x_n\conv{\sf o}{n}x
&\quad\equi\quad x_n\conv{\sigma\sf o}{n}x\\
&\quad\equi\quad \a b\in B\; \e z \;\big(\a^\infty n\; |x_n-x|\leqslant z\;\&\; b\not<z\big).
}
Thus,  $\sigma$-order convergence of sequences defines an analytic relation on $(x_n)$ and $x$.
\end{lemme}

\begin{proof}
Assume first that $x_n\conv{\sf o}{n}x$. Then there is a decreasing net $(z_\mu)$ with infimum $0$ so that every $z_\mu$ bounds all but finitely many of the expressions $|x_n-x|$. In particular, if $b\in B$, then we have that $b\not<z_\mu$ for some $\mu$, which shows that 
$$
\a b\in B\; \e z \;\big(\a^\infty n\; |x_n-x|\leqslant z\;\&\; b\not<z\big).
$$

Assume now, in turn, that $\a b\in B\; \e z \;\big(\a^\infty n\; |x_n-x|\leqslant z\;\&\; b\not<z\big)$. Enumerate $B$ as $B=\{b_1,b_2,\ldots\}$ and, for each $k$, choose some $z_k$ so that $\a^\infty n\; |x_n-x|\leqslant z_k$, whereas $b_k\not<z_k$. Let also $y_m=\bigwedge_{k=1}^mz_k$. Then $y_m\downarrow0$ and, for every $m$, we have $|x_n-x|\leqslant y_m$ for all but finitely many $n$, i.e., $x_n\conv{\sigma\sf o}{n}x$. As $\sigma$-order convergence implies order convergence, this finishes the proof.
\end{proof}

\begin{cor}
Suppose $X$ is a separable Banach lattice with a countable $\pi$-basis. Then Theorem \ref{super} holds without the additional assumption of analytic determinacy.
\end{cor}


\section{Proofs for filter bases}
In the following, we shall identify the powerset $\ku P(\N)$ with the Cantor space $\{0,1\}^\N$. If $(e_n)$ is any sequence in a Banach space $X$, we may define a Borel measurable function 
$$
\theta:X\times \R^\N\times \R_+\maps{} \ku P(\N)
$$ 
by letting, for all $x\in X$, $(a_n)\in \R^\N$ and $\eps>0$, 
\begin{equation}\label{theta}
\theta\big(x,(a_n),\eps\big)=\Big\{m\in \N\Del\, \NORM{x-\sum_{n=1}^ma_ne_n}<\eps\Big\}.
\end{equation}

Assume now that $(e_n,e^\sharp_n)$ is a fixed biorthogonal system in $X$. Then a filter $\mathcal{F}$ on $\N$ (always assumed to be proper and containing the Fréchet filter of all cofinite sets) is compatible with $(e_n,e^\sharp_n)$ if, for all $x\in X$,
$$
\sum_{n=1}^me^\sharp_n(x)e_n\conv{}{m\to \mathcal{F}}x
$$
and, for all sequences $(a_n)\in \R^\N$ other than $\vec 0=(0,0,\ldots)$, we have 
$$
\sum_{n=1}^ma_ne_n\nconv{}{m\to \mathcal{F}}0.
$$
Indeed, these two conditions taken together ensure that $\sum_{n=1}^\infty e^\sharp_n(x)e_n$ is the unique $\mathcal{F}$-expansion of an element $x\in X$. Rewriting these conditions in terms of $\theta$, we find that the filter $\mathcal{F}$ is compatible with $(e_n,e^\sharp_n)$ if and only if 
\begin{equation}
\theta\big(x,(e^\sharp_n(x)),\eps\big)\in \mathcal{F}
\end{equation}
for all $x\in X$ and $\eps>0$ and, moreover, $\mathcal{F}$ satisfies the property $\Phi$ defined by
\begin{equation}\label{phi}
\Phi(\mathcal{F}) \;\;\equi\;\; \a (a_n)\in \R^\N\setminus \{{\vec 0}\}\;\; \e k\in \N\quad\theta\big(0,(a_n),\tfrac 1k\big)\notin \mathcal{F}.
\end{equation}
To simplify notation, if $(e_n,e^\sharp_n)$ is a biorthogonal system, we let $E:X\maps{}\R^\N$ denote the biorthogonal operator $Ex=\big(e^\sharp_n(x)\big)$. In particular, $E$ is continuous if and only if all the $e^\sharp_n$ are continuous.

\begin{lemme}
Every filter basis system $(e_n,e^\sharp_n)$ for a Banach space $X$ has a smallest compatible filter. 
\end{lemme}

\begin{proof}
Observe that
\begin{equation}\label{smallest filter}
\mathcal{A}=\Big\{ a\in\ku P( \N)\Del \bigcap_{i=1}^m\theta\big(x_i,Ex_i,\eps\big) \subseteq a \text{ for some } x_i\in X \text{ and } \eps>0\Big\}
\end{equation}
is the smallest filter on $\N$ containing all images $\theta\big(x,Ex,\eps\big)$ for $x\in X$ and $\eps>0$. In particular, $\mathcal{A}$ is contained in every compatible filter.
On the other hand, if $\mathcal{F}$ is a compatible filter and $(a_n)\neq \vec 0$, then there is some $k$ so that
$$
\theta(0,(a_n),\tfrac 1k)\notin \mathcal{F}\supseteq \mathcal{A},
$$
which shows that $\Phi(\mathcal{A})$ holds and hence that $\mathcal{A}$ is a compatible filter for $(e_n,e^\sharp_n)$. 
\end{proof}

Note that, if $E$ is continuous, then the smallest compatible filter $\mathcal{A}$ (see Equation (\ref{smallest filter})) is analytic when viewed as a subset of $\ku P(\N)$. This is \cite[Theorem B]{Rancourt}. 
Observe also that, if $\mathcal{F}$ is a compatible analytic filter,  then $E$ has analytic graph,
$$
\big(x,(a_n)\big)\in \ku G E\;\equi\; \a k\in \N\;\;\; \theta\big(x,(a_n),\tfrac 1k\big)\in \mathcal{F},
$$
and thus is Borel measurable \cite[Theorem 14.12]{Kechris} and therefore continuous \cite[Theorem 9.10]{Kechris}. This is \cite[Theorem A]{Rancourt}.

\begin{proof}[Proof of Theorem \ref{thm:filter}]
Let $(e_n,e^\sharp_n)$ be a fixed filter basis system for $X=[e_n]$. We remark that the implications (4)$\saa$(1)$\saa$(3) have been noted above. Also, the implications (5)$\saa$(4) and (1)$\saa$(2) are trivial, so it suffices to show (2)$\saa$(1) and (3)$\saa$(5).

(2)$\saa$(1): 
Assume that  $(e_n)$ is minimal, i.e., that $e_k\notin [e_n]_{n\neq k}$ for all $k$. 
Since every $x\in\ker e^\sharp_k$ can be written as $x=^\mathcal{F}\sum_{n=1}^\infty a_ne_n$ with $a_k = 0$, it follows that $\ker e^\sharp_k\subseteq [e_n]_{n\neq k}$. If $e^\sharp_k$ were discontinuous, then its kernel would be dense, and so would $[e_n]_{n\neq k}$. This is a contradiction.

(3)$\saa$(5): Assume that the minimal compatible filter $\mathcal{A}$ is analytic.  We define  a binary predicate $\Psi$ on subsets of $\ku P(\N)$ by letting for $B,C\subseteq \ku P(\N)$
\maths{
\Psi(B,C) \;\;\equi\;\; 
& \a x\subseteq y\subseteq \N\; (x\in B\to y\notin C)\;\;\&\\
&\a x,y\subseteq \N \;( x,y\in B\to x\cap y\notin C)\;\;\&\\
&\a x\subseteq \N \;(x \text{ is cofinite }\to x\notin C)\;\;\&\\
& \emptyset\notin B.
}
Observe that, if $\sim\!\mathcal{F}$ denotes the complement of a set $\mathcal{F}\subseteq \ku P(\N)$, then $\Psi(\mathcal{F},\sim\! \mathcal{F})$ holds if and only if $\mathcal{F}$ is a proper filter on $\N$ containing all cofinite sets. 

Consider now the conjunction 
$$
\Gamma(B,C)\;\;\equi\;\; \Phi(B)\;\;\&\;\; \Psi(B,C),
$$
where $\Phi$ is defined as in (\ref{phi}), and observe that $\Gamma$ is a {\em hereditary} predicate in both variables, i.e., passes to subsets, and is {\em continuous upwards in the second variable}, i.e., if $C_1\subseteq C_2\subseteq \ldots$ and $\Gamma(B,C_n)$ hold for all $n$, then also $\Gamma(B,\bigcup_nC_n)$. Furthermore, $\Gamma$ is ${\bf \Pi}^1_1$ on ${\bf \Sigma}^1_1$. That is, if $Y$ is a Polish space and $B, C\subseteq Y\times \ku P(\N)$ are ${\bf \Sigma}^1_1$, then 
$$
\big\{ y\in Y\del \Gamma(B_y, C_y) \text{ holds}\big\}
$$
is ${\bf \Pi}^1_1$ (where for $B\subset Y\times \ku P(\N)$ and $y\in Y$, we denote $B_y=\{A\in \ku P(\N)\del\,(y,A)\in B\}$).

By the discussion above, we see that, if $\mathcal{A}\subseteq \mathcal{F}\subseteq\ku P(\N)$, then $\mathcal{F}$ is a compatible filter if and only if $\Gamma(\mathcal{F},\sim\! \mathcal{F})$. In particular, $\Gamma(\mathcal{A},\sim\! \mathcal{A})$ and hence, by the Second Reflection Theorem \cite[Theorem 35.16]{Kechris}, there is some Borel set $\mathcal{F}\subseteq \ku P(\N)$ so that $\mathcal{A}\subseteq \mathcal{F}$ and $\Gamma(\mathcal{F},\sim \!\mathcal{F})$. Thus, $\mathcal{F}$ is a compatible Borel filter.
\end{proof}

\begin{rem}
Observe that, if the biorthogonal operator $E:X\maps{}\R^\N$ associated with some $\mathcal{F}$-filter basis $(e_n)$ for $X$ is continuous, then the operator range $E[X]$ is the continuous injective image of a separable Banach space and is therefore a Borel linear subspace of $\R^\N$ \cite[Theorem 15.1]{Kechris}. However, if $\mathcal{F}$ is actually Borel, we have explicit bounds on the Borel complexity of $E[X]$ in terms of the Borel complexity of $\mathcal{F}$. Indeed, 
$$
(a_n)\in E[X] \;\equi\; \a l\; \e k\;       \Big\{m\in \N\Del\, \NORM{\sum_{n=1}^ma_ne_n-\sum_{n=1}^ka_ne_n}<\tfrac1l\Big\}\in \mathcal{F}.
$$
To see this, note that the implication from left to right is immediate. For the implication from right to left, note that, if $k_l$ are such that  
$$
\Big\{m\in \N\Del\, \NORM{\sum_{n=1}^ma_ne_n-\sum_{n=1}^{k_l}a_ne_n}<\tfrac1l\Big\}\in \mathcal{F}
 $$
 for all $l$, then $\big(\sum_{n=1}^{k_l}a_ne_n\big)_l$ is Cauchy and converges to some $x$ so that $Ex=(a_n)$.
\end{rem}



\section*{Acknowledgements}
C.~Rosendal~was partially supported by the U.S.~National Science Foundation under Award Numbers DMS-2246986 and DMS-2204849. A.~Avil\'{e}s was supported by MICIU/AEI /10.13039/501100011033/ and ERDF-A way of making Europe (project PID2021-122126NB-C32). A.~Avil\'{e}s and P.~Tradacete were supported by Fundaci\'{o}n S\'{e}neca - ACyT Regi\'{o}n de Murcia. P.~Tradacete was partially supported by grants PID2020-116398GB-I00, PID2024-162214NB-I00 and CEX2023-001347-S funded by MCIN/AEI/10.13039/501100011033, as well as by a 2022 Leonardo Grant for Researchers and Cultural Creators, BBVA Foundation.



\end{document}